\documentclass{amsart}

\usepackage{amssymb,amsmath,amsthm,latexsym,amscd}
\usepackage{epsfig}
\usepackage{psfrag}

\newtheorem{Theorem}{\bf Theorem}
\newtheorem{lemma}[Theorem]{\bf Lemma}
\newtheorem{proposition}[Theorem]{\bf Proposition}

\newtheorem{example}[Theorem]{\bf Example}
\newtheorem{remark}[Theorem]{\bf Remark}
\newtheorem{theorem}[Theorem]{\bf Theorem}

\def\qed{\hfill$\Box$}
\newcommand{\be}{\begin{equation}}
\newcommand{\ee}{\end{equation}}

\def\hpic #1 #2 {\mbox{$\begin{array}[c]{l} 
\epsfig{file=#1,height=#2}\end{array}$}}
\def\wpic #1 #2 {\mbox{$\begin{array}[c]{l} 
\epsfig{file=#1,width=#2}\end{array}$}}

\begin{document}

\title[Generators for subfactor planar algebras]{A note on generators for finite depth subfactor planar algebras}

\author{Vijay Kodiyalam}
\address{The Institute of Mathematical Sciences, Chennai, India}
\email{vijay@imsc.res.in,stupurani@gmail.com}
\author{Srikanth Tupurani}



\begin{abstract} 
We show that a subfactor planar algebra of finite depth $k$
is generated by a single $s$-box, for $s \leq min\{k+4,2k\}$.
\end{abstract}
\maketitle

The main result of \cite{KdyTpr2010} shows that a subfactor planar
algebra of finite depth is singly generated with a finite presentation.
If $P$ is a subfactor planar algebra of depth $k$, it is shown there that
a single $2k$-box generates $P$. It is natural to ask what  the smallest
$s$ is such that a single $s$-box generates $P$. While we do not resolve this question completely, we show in this note
that $s  \leq min\{k+4,2k\}$ and that $k$ does not suffice in general.
All  terminology and unexplained notation will be as in \cite{KdyTpr2010}.

For the rest of the paper fix a subfactor planar algebra $P$ of finite
depth $k$. Let $2t$ be such that it is the even number of $k+3$ and $k+4$.
We will show that some $s$-box generates $P$ as a planar algebra, where $s = min\{2k,2t\}$.
%
The main observation is the following result about involutive algebra anti-automorphisms of finite-dimensional complex
semisimple algebras. We mention as a matter of terminology that we
always deal with ${\mathbb C}$-algebra anti-automorphisms and
automorphisms (as opposed to those that might induce a non-identity involution on the base field ${\mathbb C}$).
Also, as in common in Hopf algebra literature, we will
use $Sa$ instead of $S(a)$ to demote the image of $a$ under a map $S$ of algebras.

\begin{theorem}\label{gen}
Let $A$ be a finite-dimensional complex semisimple algebra and let $S:A \rightarrow A$
be an involutive algebra anti-automorphism. 
Suppose that $A$ has no $2 \times 2$ matrix summand.
Then, there exists $a \in A$ such that
$a$ and $Sa$ generate $A$ as an algebra.
\end{theorem}

Before beginning the proof of this theorem, we observe that the somewhat peculiar restriction on $A$ not having an
$M_2({\mathbb C})$ summand is really necessary.

\begin{remark}
The map $S : M_2({\mathbb C}) \rightarrow M_2({\mathbb C})$ defined by $Sa = adj(a)$
is easily verified to be an involutive algebra anti-automorphism, while
there exists no $a \in M_2({\mathbb C})$ that together with $Sa$ generates
$M_2({\mathbb C})$ since these generate only a commutative
subalgebra.
\end{remark}

We pave the way for a proof of Theorem \ref{gen} by studying the two special cases when $A = M_n({\mathbb C})$ and $A = M_n({\mathbb C}) \oplus
M_n({\mathbb C})$. In these, $n$ is a fixed positive integer.
We will need the following lemmas that specify a `standard form' for each of these two special cases.

\begin{lemma}\label{fixed}
Let $S$ be an involutive algebra anti-automorphism of  $M_n({\mathbb C})$. There is an algebra automorphism of  $M_n({\mathbb C})$ under which  $S$ is identified with either (i) the transpose map or (ii)  the transpose map followed by conjugation by the matrix
$$
J = \left[
\begin{array}{cc}
0 & I_k\\
-I_k & 0
\end{array}
\right] ( = -J^T = -J^{-1}).
$$
The second case may arise only when $n = 2k$ is even (and $I_k$
denotes, of course, the identity matrix of size $k$).
\end{lemma}

\begin{proof} 
Let $T$ denote the transpose
map on $M_n({\mathbb C})$. The composite map $TS$ is then an algebra automorphism of $M_n({\mathbb C})$ and is consequently given by
conjugation with an invertible matrix, say $u$. Thus $Sx = (uxu^{-1})^T$.
Involutivity of $S$ implies that $u$ is either symmetric or skew-symmetric.  By Takagi's factorization (see p204 and p217 of \cite{HrnJhn1990}), $u$ is of the form $v^Tv$
if it is symmetric and of the form $v^TJv$ if it is skew-symmetric, for
some invertible $v$. For the algebra automorphism of $M_n({\mathbb C})$
 given by conjugation with $v$, $S$ gets identified in the symmetric case with the transpose map and in the skew-symmetric case with the  transpose map followed by conjugation by $J$.
\end{proof}

\begin{lemma}\label{invol}
Let $S$ be an involutive algebra anti-automorphism of $M_n({\mathbb C}) \oplus
M_n({\mathbb C})$ that interchanges the two minimal central projections. There is an algebra automorphism of  $M_n({\mathbb C}) \oplus
M_n({\mathbb C})$ fixing the minimal central projections under which  $S$ is identified with the map $x \oplus y \mapsto y^T \oplus x^T$.
\end{lemma}

\begin{proof}
The map $x \oplus y \mapsto S(y^T \oplus x^T)$ is an algebra automorphism
of $M_n({\mathbb C}) \oplus
M_n({\mathbb C})$ fixing the minimal central projections and is therefore given by $x \oplus y \mapsto uxu^{-1} \oplus vyv^{-1}$
for invertible $u,v$. Hence $S(x \oplus y) = uy^Tu^{-1} \oplus vx^Tv^{-1}$.

Thus, $S^2(x \oplus y) = u(v^{-1})^Txv^Tu^{-1} \oplus v(u^{-1})^Tyu^Tv^{-1}$.
Involutivity of $S$ now implies that $v^Tu^{-1}$ and $u^Tv^{-1}$ are both
scalar matrices, or equivalently, $v^T = \lambda u$ and $u^T = \mu v$ for
non-zero scalars $\lambda, \mu$. Taking transposes shows that $\lambda \mu = 1$ and finally, by replacing $u$ by $\lambda u$,
 we may assume that $v = u^T$. 
%

The commutativity of the following diagram:
\begin{center}
$\begin{CD}
M_n({\mathbb C}) \oplus
M_n({\mathbb C}) @>x \oplus y \mapsto u^{-1}xu \oplus y>> M_n({\mathbb C}) \oplus
M_n({\mathbb C})\\
@VVSV @VVx \oplus y \mapsto y^T \oplus x^TV\\
M_n({\mathbb C}) \oplus
M_n({\mathbb C}) @>x \oplus y \mapsto u^{-1}xu \oplus y>> M_n({\mathbb C}) \oplus
M_n({\mathbb C})
\end{CD}$
\end{center}
now implies that under the algebra automorphism of $M_n({\mathbb C}) \oplus
M_n({\mathbb C})$ given by $x \oplus y \mapsto u^{-1}xu \oplus y$, $S$
is identified with $x \oplus y \mapsto y^T \oplus x^T$.
\end{proof}

The proof of Theorem \ref{gen} in the case $A = M_n({\mathbb C})$ (for $n \neq 2$) needs some preparation. For a subset $S \subseteq M_n({\mathbb C})$ we use the notation $S^\prime$, as usual, to denote
its commutant in $M_n({\mathbb C})$.

\begin{lemma}\label{nonempty}
If $U \subseteq {\mathbb C}^{2N}$ is non-empty and Zariski open then $$U \cap \{(z_1,\cdots,z_N,\overline{z_1},\cdots,\overline{z_N}) : z_i \in {\mathbb C}\} \neq \emptyset.$$
\end{lemma}

\begin{proof} It suffices to see that $S = \{(z_1,\cdots,z_N,\overline{z_1},\cdots,\overline{z_N}) : z_i \in {\mathbb C}\}$ is Zariski dense in ${\mathbb C}^{2N}$. If a polynomial $f$ in $2N$ variables vanishes on $S$, then the polynomial $p(u_1,\cdots,u_N,v_1,\cdots v_N) = f(u_1+iv_1,\cdots,u_N+iv_N,
u_1-iv_1,\cdots,u_N-iv_N)$ vanishes on ${\mathbb R}^{2N}$. It is then easily seen by induction on the number of variables that $p$ identically vanishes and then, so does $f$.
\end{proof}

\begin{proposition}\label{open}
For $n > 1$, the set
\begin{eqnarray*}
U = \left\{ (P,Q) \in M_n({\mathbb C}) \times
M_n({\mathbb C}):   P,Q {\mbox {~invertible and~}}  
\left.
\left.
\begin{array}{cc}
 & \\
 & 
\end{array}
\right.
\right.^{}
\right. \\ 
\left.
\left\{
\left[
\begin{array}{cc}
0 & P\\
Q & 0
\end{array}
\right],
\left[
\begin{array}{cc}
0 & P^T\\
Q^T & 0
\end{array}
\right]
\right\}^\prime = {\mathbb C}I_{2n}
\right\}.
\end{eqnarray*}
is a non-empty Zariski open subset of $M_n({\mathbb C}) \times
M_n({\mathbb C}).$
\end{proposition}

\begin{proof}
For an arbitrary matrix
$
\left[
\begin{array}{cc}
X & Y\\
Z & W
\end{array}
\right] \in M_{2n}({\mathbb C})
$,
the condition that it commute with both 
$\left[
\begin{array}{cc}
0 & P\\
Q & 0
\end{array}
\right]$ and $
\left[
\begin{array}{cc}
0 & P^T\\
Q^T & 0
\end{array}
\right]$
is given by a set of $8n^2$ homogeneous linear equations in the $4n^2$ entries of $X,Y,Z,W$ with coefficients
(linear) polynomials in the entries of $P$ and $Q$. 

The solution space for this system is at least one dimensional (since it certainly contains the identity matrix) and thus the coefficient matrix has rank at most $4n^2-1$. The condition that the solution space is exactly one dimensional is hence equivalent to the condition that the coefficient matrix has rank at least $4n^2-1$, which
is clearly Zariski open condition in the entries of $P$ and $Q$. It follows that $U$ is Zariski open.

%
%

To show non-emptiness of $U$, choose an invertible $Q \in M_n({\mathbb C})$ such that $Q$ and $Q^T$ generate $M_n({\mathbb C})$ as an algebra. For instance,
$Q$ could be $I_n + N_n$ where $N_n$ is the $n \times n$ nilpotent matrix with super-diagonal entries all $1$ and $0$ entries elsewhere. The condition that
$
\left[
\begin{array}{cc}
X & Y\\
Z & W
\end{array}
\right] \in M_{2n}({\mathbb C})
$ commutes with both 
$\left[
\begin{array}{cc}
0 & I\\
Q & 0
\end{array}
\right]$ and $
\left[
\begin{array}{cc}
0 & I\\
Q^T & 0
\end{array}
\right]$
is equivalent to the set of equations:
\begin{eqnarray*}
&YQ=QY=Z=YQ^T=Q^TY\\
&WQ=QX,~X=W,~WQ^T=Q^TX.
\end{eqnarray*}
Since $Y$ commutes with $Q$ and $Q^T$ (which generate $M_n({\mathbb C})$),  $Y = \lambda I_n$ for a scalar $\lambda \in {\mathbb C}$. Thus
$Z = \lambda Q = \lambda Q^T$. Now (and this is the crucial point where
$n>1$ is needed), since $Q$ and $Q^T$ generate $M_n({\mathbb C})$ which is not commutative, they cannot be equal and so $\lambda=0$. Since $X=W$ and hence commutes with both $Q$ and $Q^T$, $X=W=\mu I$ for some scalar $\mu \in {\mathbb C}$. Thus $(I,Q) \in U$.
\end{proof}

\begin{proposition}\label{prop1}
Let $S$ be an involutive algebra anti-automorphism of  $M_m({\mathbb C})$ with $m \neq 2$.
There exists
invertible $x \in M_m({\mathbb C})$ which, together with $Sx$, generates $M_m({\mathbb C})$ as an algebra.
\end{proposition}

\begin{proof}
First, we may assume by Lemma \ref{fixed} that $S$ is either (i) the transpose map or (ii) the transpose map followed by conjugation by $J$. In Case (i),
as in the proof of Proposition \ref{open}, $x =I_m+N_m$ is invertible and such that
$x$ and $Sx$ generate $M_m({\mathbb C})$ as an algebra.

In Case (ii), $m =2n$ is necessarily even. It then follows from Proposition \ref{open} and Lemma \ref{nonempty} that there is an invertible $P \in M_n({\mathbb C})$ such that
$$
\left\{
\left[
\begin{array}{cc}
0 & P\\
\overline{P} & 0
\end{array}
\right],
\left[
\begin{array}{cc}
0 & P^T\\
\overline{P}^T & 0
\end{array}
\right]
\right\}^\prime = {\mathbb C}I_{2n}
$$
The commutant of these two matrices is the same as that of the algebra
they generate which is a $*$-subalgebra of $M_m({\mathbb C})$ since
they are adjoints of each other. By the double commutant theorem, it follows that the algebra generated by these is the whole of $M_m({\mathbb C})$. Now take $x = 
\left[
\begin{array}{cc}
0 & P\\
\overline{P} & 0
\end{array}
\right]
$.
\end{proof}

In proving Theorem \ref{gen} for $A = M_n({\mathbb C}) \oplus M_n({\mathbb C})$, we will need the following lemma.

\begin{lemma}\label{genl}
Let $A$ and $B$ be finite dimensional complex unital algebras and
let $a \in A$ and $b \in B$ be invertible. Then, for all but finitely many
$\lambda \in {\mathbb C}$, the  algebra generated by $a \oplus \lambda b \in A \oplus B$ contains both $a~(=a \oplus 0)$ and $b~(= 0 \oplus b)$.
\end{lemma} 

\begin{proof} We may assume that $\lambda \neq 0$ and then it suffices to see that $a$ is expressible as a polynomial
in $a \oplus \lambda b$. 
Note that since $a \oplus \lambda b$ is invertible and $A \oplus B$
is finite dimensional, the algebra generated by $a \oplus \lambda b$
is actually unital. In particular, it makes sense to evaluate any complex
univariate polynomial on $a \oplus \lambda b$.

Let $p(X)$ and $q(X)$ be the minimal polynomials of
$a$ and $b$ respectively. By invertibility of $a$ and $b$, neither
$p$ nor $q$ has $0$ as a root. The minimal polynomial of $\lambda b$
is $q(\frac{X}{\lambda})$. Unless $\lambda$ is the quotient of a root of
$p$ by a root of $q$, $p(X)$ and $q(\frac{X}{\lambda})$ will have no common roots and therefore be coprime. So there will exist a polynomial $r(X)$ that is divisible by
$q(\frac{X}{\lambda})$ but is $X$ modulo $p(X)$. Thus $r(a \oplus \lambda b) = a$, as desired.
\end{proof}

\begin{proposition}\label{cor2} Let $S$ be an involutive algebra anti-automorphism of $M_n({\mathbb C}) \oplus
M_n({\mathbb C})$ that interchanges the two minimal central projections.
There exists invertible $x \oplus y \in M_n({\mathbb C}) \oplus
M_n({\mathbb C})$ which together with $S(x \oplus y)$ generates $M_n({\mathbb C})   \oplus
M_n({\mathbb C})$ as an algebra.
\end{proposition}

\begin{proof} First, by Lemma \ref{invol}, we may assume that
$S$ is the map $x \oplus y \mapsto y^T \oplus x^T$. Now, as in the proof of Proposition \ref{prop1}, there is an invertible $x \in M_n({\mathbb C})$ such that $x$ and $x^T$ generate $M_n({\mathbb C})$.
By Lemma~\ref{genl}, for all but finitely many $\lambda \in {\mathbb C}$,
the  algebra generated by $x \oplus \lambda x$ contains
$x \oplus 0$ and $0 \oplus x$ and similarly the  algebra generated by $\lambda x^T \oplus x^T$ contains $x^T \oplus 0$ and $0 \oplus x^T$.
Thus the algebra generated by $x \oplus \lambda x$ and $\lambda x^T \oplus x^T$ is the whole of $M_n({\mathbb C})   \oplus
M_n({\mathbb C})$.
%
%
%
%
\end{proof}

\begin{proof}[Proof of Theorem \ref{gen}] Let $\hat{A}$ denote the (finite) set of all inequivalent irreducible representations of $A$
and for $\pi \in \hat{A}$, let 
$d_\pi$ denote its dimension.
Since $S$ is an involutive anti-automorphism, it acts as an involution on the set of minimal central projections of $A$.
It is then easy to see
that there exist 
subsets $\hat{A}_1$ and $\hat{A}_2$ of $\hat{A}$ and an
identification
$$
A \rightarrow \bigoplus_{\pi \in \hat{A}_1} M_{d_\pi}({\mathbb C}) \oplus 
\bigoplus_{\pi \in \hat{A}_2} (M_{d_\pi}({\mathbb C}) \oplus M_{d_\pi}({\mathbb C}))
$$
such that 
%
%
%
each summand is $S$-stable.

Now, by Propositions  \ref{prop1} and \ref{cor2}, in each summand of the above decomposition, either $M_{d_\pi}({\mathbb C})$ or $M_{d_\pi}({\mathbb C}) \oplus M_{d_\pi}({\mathbb C})$, there is an invertible element which together with its image under $S$ generates that summand.

Finally, an inductive application of Lemma \ref{genl} shows that if $a$
is a
general linear combination of these generators, then $a$ and $Sa$  generate $A$ as an algebra.
\end{proof}

Before we prove our main result, we will need a result
about connected pointed bipartite graphs. Recall that a bipartite graph has its vertex set partitioned into `even' and
`odd' vertices and all its edges connect an even and an odd vertex. It is pointed if a certain even vertex, normally denoted by $*$, is distinguished. Its depth is the largest
distance of a vertex from $*$.

\begin{proposition}\label{graph}
Let $\Gamma$ be a connected pointed bipartite graph of depth $k \geq 3$. For any vertex $v$ of $\Gamma$, let
$t$ be the one of $k+3,k+4$ with the same parity as
$v$. The number of paths of length $t$ from $*$ to $v$
is at least $3$.
\end{proposition}

\begin{proof}
We analyse three cases depending on the distance of $v$ from $*$.\\
Case I: If $v=*$, note that $t \geq 6$ is even. To show that there are at least 3 paths of length $t$ from $*$ to $*$, it suffices to show that there are at least 3 paths
of length 6 from $*$ to $*$. Since $k \geq 3$, choose any vertex at distance 2 from $*$ and a path from $*$ to the chosen vertex. It is easy to see that there are at least 3 paths of length 6 from $*$ to $*$ 
supported on the edges of this path.\\
%
Case II: If $v$ is at distance 1 from $*$, then $t \geq 7$ is odd. As observed in Case I, there are at least 3 paths of length 6 from $*$ to $*$ and consequently
at least 3 paths of length 7 from $*$ to $v$.\\
Case III: Suppose $v$ is at a distance $n$ from $*$, where $n>1$. 
Observe that if $n$ and $k$ have the same parity, then $n \leq k$ while in the other case, $n \leq k-1$.
Choose a path $\xi_{1} \xi_{2} \xi_{3} \cdots \xi_{n}$ from $*$ to $v$.
Then $\xi_2 \neq \overline{\xi_1}.$ 
Then we have three paths 
$\xi_{1} \overline{\xi_{1}} \xi_{1} \overline{\xi_{1}} \xi_{1} \xi_{2} \cdots \xi_{n}$,
 $\xi_{1} \xi_{2} \overline{\xi_2} \xi_{2} \overline{\xi_{2}} \xi_{2} \cdots \xi_{n}$, and
 $\xi_{1} \overline{\xi_{1}} \xi_{1} \xi_{2} \overline{\xi_2} \xi_{2} \cdots \xi_{n}$ of length $n+4$ from $*$ to $v$. 
Thus if $n$ and $k$ have the same parity, so that $t=k+4$, then there exist at least 3 paths of length $t$ from $*$ to $v$.
If $n$ and $k$ have opposite parity then $t = k+3$ and since $n \leq k-1$ in this case, since there exist at least 3 distinct paths of length $n+4$ from
$*$ to $v$, there also exist 3 distinct paths of length $t$ from $*$ to $v$.
%
%
\end{proof}

We now prove the main result of this note.

\begin{theorem}
Let $P$ be a subfactor planar algebra of finite depth $k$. Let $2t$ be
the even number in $\{k+3, k+4\}$. Let $s=min\{2k,2t\}$. Then $P$ is generated by a single $s$-box.
\end{theorem}

\begin{proof}
Case I: If $k \leq 3$, $s=2k$. Then by Proposition 5.1 of \cite{KdyTpr2010}, $P$ is generated by a single $s$ box.

Case II: If $k>3$, so that $s=2t$, let $\Gamma$ be the principal graph of the subfactor planar algebra $P$. Then from Proposition \ref{graph}, the number of paths of length $s$ from the $*$-vertex to any even vertex $v$ in $\Gamma$ is at least 3. So $P_{s}$ does not have an $M_{2}(\mathbb{C})$ summand.
Consider the $t^{\mbox {\tiny{th}}}$ power, say $X$, of the $s$-rotation tangle. This tangle changes the position of $*$ on an $s$-box from the top left to the bottom right position. Clearly $Z^P_X :P_{s} \to P_{s}$ is an involutive algebra anti-automorphism. 
From Theorem \ref{gen}, there exists an element $a \in P_s$ such that $a$ and $Z^P_X(a)$ generate $P_{s}$ as an algebra.
Since $s \geq k$, the planar algebra generated by $P_s$ contains $P_k$ and thus is the whole of $P$.
Hence the single $s-$box containing $a$ generates the planar algebra $P$.
%
\end{proof}

We finish by showing that $k+1$ might actually be needed.

\begin{example} Let $P = P(V)$ be the tensor planar algebra (see \cite{Jns1999}) for details) of a vector space
$V$ of dimension greater than $1$. It is easy to see that $depth(P) =1$.
However, given any $a \in P_1 = End(V)$, if $Q$ is the planar subalgebra of $P$
generated by $a$, a little thought shows that  $Q_1$ is just the algebra generated by $a$ and
is hence abelian while $P_1$ is not.
\end{example}



\section*{Acknowledgments}
We are grateful to Prof. T. Y. Lam for his remarks and to the referee for a
very careful reading.

\end{document}